\documentclass[12pt,reqno]{amsart}

\usepackage{amssymb}

\numberwithin{equation}{section}

\newtheorem{theorem}{Theorem}[section]
\newtheorem{proposition}[theorem]{Proposition}

\theoremstyle{definition}

\begin{document}

\baselineskip=15pt

\title[Connections on curves defined over a number field]{Connections and Higgs bundles
on curves defined over a number field}

\author[I. Biswas]{Indranil Biswas}

\address{Department of Mathematics, Shiv Nadar University, NH91, Tehsil Dadri,
Greater Noida, Uttar Pradesh 201314, India}

\email{indranil.biswas@snu.edu.in, indranil29@gmail.com}

\author[S. Gurjar]{Sudarshan Gurjar}

\address{Department of Mathematics, Indian Institute of Technology Bombay,
Powai, Mumbai 400076, Maharashtra, India}

\email{sgurjar@math.iitb.ac.in}



\date{}

\begin{abstract}
Let $X_0$ be an irreducible smooth projective curve defined over $\overline{\mathbb Q}$
and $\mathbb E$ a vector bundle on $X_0$. We give a criterion for connections on the
base change ${\mathbb E}\otimes_{\overline{\mathbb Q}}{\mathbb C}\, \longrightarrow\,
X_0\times_{{\rm Spec}\,\overline{\mathbb Q}} {\rm Spec}\,\mathbb{C}
$ to $\mathbb C$ to be the base change of some
connection on $\mathbb E$. A similar criterion is given for Higgs fields on
${\mathbb E}\otimes_{\overline{\mathbb Q}}{\mathbb C}$.
\end{abstract}

\maketitle

\section{Introduction}

A well-known theorem of Gennadii V. Belyi says that an irreducible smooth complex projective curve $X$ is 
isomorphic to one defined over $\overline{\mathbb Q}$ if and only if $X$ admits a nonconstant morphism to
$\mathbb{P}^1_{\mathbb C}$ whose branch locus is contained in the subset $\{0,\,1,\, \infty\} \, \subset\,
\mathbb{P}^1_{\mathbb C}$ \cite{Be}. See \cite{Gr2}, \cite{SL} for a program inspired by the work of Belyi.

Let $X_0$ be an irreducible smooth projective curve defined over $\overline{\mathbb Q}$, and let
$X\,=\, X_0\times_{{\rm Spec}\,\overline{\mathbb Q}} {\rm Spec}\,\mathbb{C}$ be its base change to
$\mathbb C$. Given a vector bundle $E$ on $X$, we may ask when is it the base change to $\mathbb C$ of
a vector bundle on $X_0$? To answer this, let
$$
f_0\ :\ X_0\ \longrightarrow\ {\mathbb P}^1_{\overline{\mathbb Q}}
$$
be a map which is unramified over the complement ${\mathbb P}^1_{\mathbb C}\setminus
\{0,\,1,\, \infty\}$, and let
$$
f\ :\ X\ \longrightarrow\ {\mathbb P}^1_{\mathbb C}
$$
be its base change to $\mathbb C$. Consider the vector bundle $f_*E$ on ${\mathbb P}^1_{\mathbb C}$. It has
natural parabolic structure over $\{0,\,1,\, \infty\}$ \cite{AB}. The
vector bundle $E$ is the base change to $\mathbb C$ of a vector bundle on $X_0$ if and only if 
the parabolic vector bundle $f_*E$ is the base change to $\mathbb C$ of a parabolic vector
bundle on ${\mathbb P}^1_{\overline{\mathbb Q}}$ \cite{BG}.

Here we address the following question. Let $\mathbb E$ be a vector bundle on $X_0$, and let
$E\, :=\, {\mathbb E}\otimes_{\overline{\mathbb Q}}{\mathbb C}$ be the corresponding vector bundle on $X$.
What is the criterion for a connection on $E$ to be given by a connection on $\mathbb E$?
In similar vein, what is the criterion for a Higgs field on $E$ to be given by a
Higgs field on $\mathbb E$?

A connection $D$ on $E$ induces a logarithmic connection on the direct image $f_*E$. The singular
locus of the connection is the branch locus of $f$. This
induced logarithmic connection is denoted by $f_*D$. We prove the following (see Theorem \ref{thm1}):

\begin{theorem}\label{thm0}
Let $\mathbb E$ be a vector bundle on $X_0$, and let $D$ be a connection on the base change
$E\, \longrightarrow\, X$ of $\mathbb E$ to $\mathbb C$. Then there is a connection on $\mathbb{E}$ whose base
change, to $\mathbb C$, coincides with $D$ if and only if there is a logarithmic connection
$$
D_0\ :\ (f_0)_*\mathbb{E} \ \longrightarrow\ (f_0)_*\mathbb{E}\otimes K_{\mathbb{P}^1_{\overline{\mathbb Q}}}
\otimes {\mathcal O}_{\mathbb{P}^1_{\overline{\mathbb Q}}}(0+1+\infty)
$$
whose base change, to $\mathbb C$, is the connection $f_*D$.
\end{theorem}

A similar criterion is proved for Higgs fields on $E$; see Theorem \ref{thm2}.

\section{Pullback and direct image of connections and Higgs fields}

\subsection{Connections and Higgs bundles}

Let $\mathbb{X}$ be an irreducible smooth projective curve defined over an
algebraically closed field $k$ of characteristic zero.
The canonical line bundle of $\mathbb{X}$ will be denoted by $K_{\mathbb{X}}$. In fact,
the canonical line bundle of any smooth curve $M$ will be denoted by $K_M$. Let
$\mathbb{E}$ be an algebraic vector bundle over $\mathbb{X}$.

An \textit{algebraic connection} on $\mathbb{E}$ is an algebraic differential operator
$$
D\ :\ \mathbb{E}\ \longrightarrow\ \mathbb{E}\otimes K_{\mathbb{X}}
$$
of order one satisfying the Leibniz identity which says that
$D(hs)\,=\, h\cdot D(s)+ s\otimes dh$ for any locally defined algebraic section
$s$ of $\mathbb{E}$ and any locally defined algebraic function $h$ on $\mathbb{X}$.

Since all connections considered here will be algebraic, often they will be referred
to as connections dropping the term algebraic.

A \textit{Higgs field} on $\mathbb{E}$ is an algebraic section of the vector
bundle $\text{End}(\mathbb{E})\otimes K_{\mathbb{X}}$. A \textit{Higgs bundle}
on $\mathbb{X}$ is a vector bundle on $\mathbb{X}$ equipped with a Higgs field.
See \cite{Hi}, \cite{Si} for Higgs bundles.

Any vector bundle $\mathbb{E}$ admits a Higgs field, but not every vector bundle
admits a connection.

A vector bundle $V$ on $\mathbb{X}$ is called \textit{decomposable} if $V\,=\, V_1\oplus V_2$,
where both $V_1$ and $V_2$ are vector bundles on $\mathbb{X}$ of positive rank. A vector bundle on
$\mathbb{X}$ is called \textit{indecomposable} if it is not decomposable.
Any vector $V$ on $\mathbb{X}$ can be expressed as a direct sum of indecomposable vector
bundles. If
\begin{equation}\label{a1}
\bigoplus_{i=1}^a V_i\ =\ V \ =\ \bigoplus_{j=1}^b W_j
\end{equation}
are two decompositions of $V$ into direct sums of indecomposable vector bundles, then a theorem
of Atiyah says that $a\,=\, b$ and there is permutation $\sigma$ of $\{1,\, \cdots,\, a\}$ such
that $V_i$ is isomorphic to $W_{\sigma(i)}$ for all $i \, \in\, \{1,\, \cdots,\, a\}$
(see \cite[p.~315, Theorem 3]{At1}). A theorem of Atiyah and Weil says that $V$ in \eqref{a1}
admits a connection if $\text{degree}(V_i)\,=\, 0$ for all $i \, \in\, \{1,\, \cdots,\, a\}$
\cite[p.~203, Theorem 10]{At2}, \cite{We}.

\subsection{Logarithmic connections and Higgs bundles}

Fix a reduced effective divisor
\begin{equation}\label{a2}
S\,=\, \sum_{i=1}^d x_i
\end{equation}
on $\mathbb{X}$; so $\{x_1,\, \cdots, \, x_d\}$ are distinct $d$ points of $\mathbb{X}$. As before,
$\mathbb{E}$ is an algebraic vector bundle over $\mathbb{X}$.

A \textit{logarithmic connection} on $\mathbb{E}$ singular over $S$ is an algebraic differential operator
$$
D\ :\ \mathbb{E}\ \longrightarrow\ \mathbb{E}\otimes K_{\mathbb{X}}\otimes {\mathcal O}_{\mathbb{X}}(S)
$$
of order one such that $D(hs)\,=\, h\cdot D(s)+ s\otimes dh$ for any locally defined algebraic section
$s$ of $\mathbb{E}$ and any locally defined algebraic function $h$ on $\mathbb{X}$.

A \textit{logarithmic Higgs field} on $\mathbb{E}$ singular on $S$ is an algebraic section of the vector
bundle $\text{End}(\mathbb{E})\otimes K_{\mathbb{X}}\otimes {\mathcal O}_{\mathbb{X}}(S)$.

\subsection{The pullback operation}

Let $\mathbb{Y}$ be an irreducible smooth projective curve over $k$ and
$$
\phi\ :\ \mathbb{Y}\ \longrightarrow\ \mathbb{X}
$$
a nonconstant morphism such that the restriction
\begin{equation}\label{b1}
\phi\big\vert_{\mathbb{Y}\setminus \phi^{-1}(S)}\ :\ \mathbb{Y}\setminus \phi^{-1}(S)
\ \longrightarrow\ \mathbb{X} \setminus S
\end{equation}
is unramified. Let
$$
\{y_1,\, \cdots,\, y_\delta\} \ :=\ \phi^{-1}(\{x_1,\, \cdots, \, x_d\}) \ \subset\ \mathbb{Y}
$$
be the set-theoretic inverse image (see \eqref{a2}); so $\{y_1,\, \cdots,\, y_\delta\}$ are distinct
$\delta$ points of $\mathbb{Y}$. Let
$$
{\mathbb S}\,=\, \sum_{i=1}^\delta y_i
$$
be the reduced effective divisor on $\mathbb{Y}$. Note that
\begin{equation}\label{a3}
\phi^* (K_{\mathbb{X}}\otimes {\mathcal O}_{\mathbb{X}}(S)) \ = \
K_{\mathbb{Y}}\otimes {\mathcal O}_{\mathbb{Y}}({\mathbb S}),
\end{equation}
because $\phi\big\vert_{\mathbb{Y}\setminus \phi^{-1}(S)}$ in \eqref{b1} is unramified.

Take a logarithmic connection
$$
D\ :\ \mathbb{E}\ \longrightarrow\ \mathbb{E}\otimes K_{\mathbb{X}}\otimes {\mathcal O}_{\mathbb{X}}(S).
$$
Using \eqref{a3} we have
$$
\phi^{-1}(D)\ :\ \phi^{-1}(\mathbb{E})\ \longrightarrow\ \phi^*(\mathbb{E}\otimes K_{\mathbb{X}}\otimes {\mathcal O}_{\mathbb{X}}(S))
\ =\ (\phi^*\mathbb{E})\otimes K_{\mathbb{Y}}\otimes {\mathcal O}_{\mathbb{Y}}({\mathbb S}).
$$
Consider the natural inclusion map $\phi^{-1}(\mathbb{E})\, \hookrightarrow\, \phi^*\mathbb{E}\,=\,
(\phi^{-1}(\mathbb{E}))\otimes_{\phi^{-1}({\mathcal O}_{\mathbb{X}})}
{\mathcal O}_{\mathbb{Y}}$. Using the Leibniz identity,
the above homomorphism $\phi^{-1}(D)$ extends uniquely to a $k$--linear homomorphism of sheaves
$$
\phi^* D\, :\, \phi^* \mathbb{E}\, =\, (\phi^{-1}(\mathbb{E}))\otimes_{\phi^{-1}({\mathcal O}_{\mathbb{X}})}
{\mathcal O}_{\mathbb{Y}}\, \longrightarrow\, \phi^*(\mathbb{E}\otimes K_{\mathbb{X}}\otimes {\mathcal O}_{\mathbb{X}}(S))
\, =\, (\phi^*\mathbb{E})\otimes K_{\mathbb{Y}}\otimes {\mathcal O}_{\mathbb{Y}}({\mathbb S}).
$$
So $\phi^* D$ is a logarithmic connection on $\phi^* \mathbb{E}$ singular on $\mathbb S$.

Take a logarithmic Higgs field $\theta\, \in\, H^0(\mathbb{X},\, \text{End}(\mathbb{E})\otimes K_{\mathbb{X}}\otimes
{\mathcal O}_{\mathbb{X}}(S))$ on $\mathbb{E}$. Using \eqref{a3},
$$
\phi^*\theta\, \in\, H^0(\mathbb{Y},\, \text{End}(\phi^*\mathbb{E})\otimes \phi^*(K_{\mathbb{X}}\otimes
{\mathcal O}_{\mathbb{X}}(S)))\,=\, H^0(\mathbb{Y},\, \text{End}(\phi^*\mathbb{E})\otimes
K_{\mathbb{Y}}\otimes {\mathcal O}_{\mathbb{Y}}({\mathbb S})).
$$
So $\phi^*\theta$ is logarithmic Higgs field on $\phi^*\mathbb{E}$ singular on $\mathbb S$.

\subsection{The direct image}

Let $\mathbf Z$ be an irreducible smooth projective curve over $k$ and
$$
\psi\ :\ \mathbb{X}\ \longrightarrow\ \mathbf{Z}
$$
a nonconstant morphism such that the restriction
\begin{equation}\label{b2}
\psi\big\vert_{{\mathbb X}\setminus S}\ :\ \mathbb{X}\setminus S\ \longrightarrow\ \mathbf{Z}
\end{equation}
is unramified. Consider the subset $\psi(\{x_1,\, \cdots, \, x_d\})\,=\, \{z_1,\, \cdots,\, z_n\}\, \subset
\,\mathbf{Z}$ (see \eqref{a2}), so $\{z_1,\, \cdots,\, z_n\}$ are distinct $n$ points of $\mathbf{Z}$. Let
$$
{\mathcal S}\ =\ \sum_{i=1}^n z_i
$$
be the reduced effective divisor on $\mathbf{Z}$. As in \eqref{a3}, we have
\begin{equation}\label{a4}
\psi^* (K_{\mathbf{Z}}\otimes {\mathcal O}_{\mathbf{Z}}({\mathcal S})) \ = \
K_{\mathbb{X}}\otimes {\mathcal O}_{\mathbb{X}}(\psi^{-1}({\mathcal S}))\ \supset\
K_{\mathbb{X}}\otimes {\mathcal O}_{\mathbb{X}}(S),
\end{equation}
where $\psi^{-1}({\mathcal S})$ is the scheme-theoretic inverse image of $\mathcal S$; this
is because the map in \eqref{b2} is unramified.

Take a logarithmic connection
$$
D\ :\ \mathbb{E}\ \longrightarrow\ \mathbb{E}\otimes K_{\mathbb{X}}\otimes {\mathcal O}_{\mathbb{X}}(S).
$$
{}From \eqref{a4} we see that $D$ gives a homomorphism
\begin{equation}\label{a5}
D\ :\ \mathbb{E}\ \longrightarrow\ \mathbb{E}\otimes K_{\mathbb{X}}\otimes 
{\mathcal O}_{\mathbb{X}}(\psi^{-1}({\mathcal S}));
\end{equation}
this homomorphism is also denoted by $D$ to have notational simplicity. By the projection formula and \eqref{a4},
\begin{equation}\label{a5b}
\psi_*(\mathbb{E}\otimes K_{\mathbb{X}}\otimes{\mathcal O}_{\mathbb{X}}(\psi^{-1}({\mathcal S})))\ =\
\psi_*(\mathbb{E}\otimes\psi^* (K_{\mathbf{Z}}\otimes{\mathcal O}_{\mathbf{Z}}({\mathcal S})))
\ = \ (\psi_*\mathbb{E})\otimes K_{\mathbf{Z}}\otimes {\mathcal O}_{\mathbf{Z}}({\mathcal S}).
\end{equation}
Consequently, $D$ in \eqref{a5} produces a homomorphism
\begin{equation}\label{a6}
\psi_* D\ :\ \psi_* \mathbb{E}\ \longrightarrow\ (\psi_*\mathbb{E})\otimes K_{\mathbf{Z}}\otimes
{\mathcal O}_{\mathbf{Z}}({\mathcal S}).
\end{equation}
It is straightforward to check that $\psi_* D$ in \eqref{a6} satisfies the Leibniz identity.
Hence $\psi_* D$ is a logarithmic connection on $\psi_* \mathbb{E}$ singular on ${\mathcal S}$.

Take a logarithmic Higgs field $\theta\, :\, \mathbb{E}\ \longrightarrow\ \mathbb{E}\otimes K_{\mathbb{X}}
\otimes {\mathcal O}_{\mathbb{X}}(S)$ on $\mathbb{E}$. From \eqref{a4}, \eqref{a5b} and the projection
formula we have
\begin{equation}\label{a6b}
\psi_*\theta \ :\ \psi_* \mathbb{E}\ \longrightarrow\ 
\psi_*(\mathbb{E}\otimes\psi^* (K_{\mathbf{Z}}\otimes{\mathcal O}_{\mathbf{Z}}({\mathcal S})))
\ = \ (\psi_*\mathbb{E})\otimes K_{\mathbf{Z}}\otimes {\mathcal O}_{\mathbf{Z}}({\mathcal S}).
\end{equation}
So $\psi_*\theta$ is a logarithmic Higgs field on $\psi_* \mathbb{E}$ singular on $\mathcal S$.

\section{Curves defined over a number field}

Let $X_0$ be an irreducible smooth projective curve defined over the algebraic
closure $\overline{\mathbb Q}$ of the field of rational numbers. Let
\begin{equation}\label{e1}
f_0\, :\, X_0\, \, \longrightarrow\, \mathbb{P}^1_{\overline{\mathbb Q}}
\end{equation}
be a nonconstant morphism which is unramified over the complement
$\mathbb{P}^1_{\overline{\mathbb Q}}\setminus \{0,\,1,\, \infty\}$; it was
proved by Belyi that such a morphism $f_0$ exists \cite{Be}. The condition on
$f_0$ means that the branch locus of $f_0$ is contained in the subset
$\{0,\,1,\, \infty\}\, \subset\, \mathbb{P}^1_{\overline{\mathbb Q}}$.

Take a vector bundle $\mathbb{E}$ on $X_0$, and consider the direct image $(f_0)_*
\mathbb{E}\,\longrightarrow\, \mathbb{P}^1_{\overline{\mathbb Q}}$. The vector bundle
$(f_0)_*\mathbb{E}$ has a parabolic
structure over the three points $\{0,\,1,\, \infty\}$ \cite[Section~4]{AB}, and this
parabolic structure is defined over $\overline{\mathbb Q}$ \cite[Proposition 2.1]{BG}.

Let
$$
X\,=\, (X_0)_{\mathbb C}\, =\, X_0\times_{{\rm Spec}\,\overline{\mathbb Q}} {\rm Spec}\,\mathbb{C}
$$
be the base change of $X_0$ to $\mathbb C$. Let
\begin{equation}\label{e2}
f\,\,:=\,\, (f_0)_{\mathbb C}\,\, :\,\,
X\,\, \longrightarrow\, \,\mathbb{P}^1_{\overline{\mathbb Q}}\times_{{\rm Spec}\,\overline{\mathbb Q}}
{\rm Spec}\,{\mathbb C}\,\,=\,\, \mathbb{P}^1_{\mathbb C}
\end{equation}
be the base change, to $\mathbb C$, of the map $f_0$ in \eqref{e1}. Let
\begin{equation}\label{e3}
E\,:=\, \mathbb{E}\otimes_{\overline{\mathbb Q}} \mathbb{C}\, \longrightarrow\, X\,=\,
X_0\times_{{\rm Spec}\,\overline{\mathbb Q}} {\rm Spec}\,\mathbb{C}
\end{equation}
be the base change of the vector bundle $\mathbb{E}\, \longrightarrow\, X_0$ to $\mathbb C$.

The direct image $f_* E\, \longrightarrow\, \mathbb{P}^1_{\mathbb C}$ (see \eqref{e2} and \eqref{e3}) is clearly
the base change to $\mathbb C$ of $(f_0)_*\mathbb{E}\, \longrightarrow\, \mathbb{P}^1_{\overline{\mathbb Q}}$.

Denote the degree three divisor $0+1+\infty$ on $\mathbb{P}^1_{\mathbb C}$ by $S_1$. The
divisor $0+1+\infty$ on $\mathbb{P}^1_{\overline{\mathbb Q}}$ will be denoted by $S_0$.
The following proposition is evident.

\begin{proposition}\label{prop1}
\mbox{}
\begin{enumerate}
\item The direct image $f_* E\, \longrightarrow\, \mathbb{P}^1_{\mathbb C}$ has a parabolic structure
defined over $\overline{\mathbb Q}$ (recall that $f_* E$ is the base change to $\mathbb C$ of $(f_0)_*\mathbb{E}$).

\item Let $D\, :\, \mathbb{E}\, \longrightarrow\, \mathbb{E}\otimes K_{X_0}$ be a connection
on $\mathbb{E}$. Let
$$
{\mathcal D}\ := \ D\otimes_{\overline{\mathbb Q}} \mathbb{C}\ :\ E\, \longrightarrow\, E\otimes K_X
$$
be the base change of $D$ to $\mathbb C$. Then the logarithmic connection
$$
f_* {\mathcal D}\ :\ f_* E\ \longrightarrow\ (f_* E)\otimes K_X
\otimes {\mathcal O}_{\mathbb{P}^1_{\mathbb C}}(S_1)
$$
(see \eqref{a6}) is the base change to $\mathbb C$ of the logarithmic connection
$$
(f_0)_* D\ :\ (f_0)_* \mathbb{E}\ \longrightarrow\ ((f_0)_*\mathbb{E})\otimes K_{\mathbb{P}^1_{\overline{\mathbb Q}}}
\otimes {\mathcal O}_{\mathbb{P}^1_{\overline{\mathbb Q}}}(S_0),
$$
where $S_0$ (respectively, $S_1$) is the divisor $0+1+\infty$ on $\mathbb{P}^1_{\overline{\mathbb Q}}$
(respectively, $\mathbb{P}^1_{\mathbb C}$).

\item Let $\theta\, :\, \mathbb{E}\, \longrightarrow\, \mathbb{E}\otimes K_{X_0}$ be a Higgs
field on $\mathbb{E}$. Let
$$
\widetilde{\theta}\ := \ \theta\otimes_{\overline{\mathbb Q}} \mathbb{C}\ :\ E\, \longrightarrow\, E\otimes K_X
$$
be the base change of $\theta$ to $\mathbb C$. Then the logarithmic Higgs field
$$
f_* \widetilde{\theta}\ :\ f_* E\ \longrightarrow\ (f_* E)\otimes K_X
\otimes {\mathcal O}_{\mathbb{P}^1_{\mathbb C}}(S_1)
$$
(see \eqref{a6b}) is the base change to $\mathbb C$ of the logarithmic Higgs field
$$
(f_0)_* \theta\ :\ (f_0)_* \mathbb{E}\ \longrightarrow\ ((f_0)_*\mathbb{E})\otimes K_{\mathbb{P}^1_{\overline{\mathbb Q}}}
\otimes {\mathcal O}_{\mathbb{P}^1_{\overline{\mathbb Q}}}(S_0).
$$
\end{enumerate}
\end{proposition}

\begin{proof}
The first statement of the proposition is a re-statement of the earlier observation
that the vector bundle $(f_0)_*\mathbb{E}$ has a parabolic
structure over the three points $\{0,\,1,\, \infty\}$ \cite[Section~4]{AB}, and this
parabolic structure is defined over $\overline{\mathbb Q}$ \cite[Proposition 2.1]{BG}.

The second statement of the proposition follows immediately from the construction of
logarithmic connection in \eqref{a6}, and the third statement of the proposition
follows from the construction of logarithmic Higgs field in \eqref{a6b}.
\end{proof}

We will prove a converse of Proposition \ref{prop1}(2). Let $V$ be a vector
bundle on $X$ equipped with an algebraic connection $\mathcal D$. Consider the parabolic vector
bundle $f_* V$ on $\mathbb{P}^1_{\mathbb C}$ equipped with the logarithmic connection
$f_* {\mathcal D}$ (see \eqref{a6}). Suppose that there is a parabolic vector bundle $W$ on
$\mathbb{P}^1_{\overline{\mathbb Q}}$, and a logarithmic connection $D$ on $W$, such that
the parabolic vector bundle with logarithmic connection $(f_* V,\, f_* {\mathcal D})$ is isomorphic
to the base change, to $\mathbb C$, of the parabolic vector bundle with
logarithmic connection $(W,\, D)$. Then there is a vector bundle $E_0$ on $X_0$, and an
algebraic connection $D_0$ on $E_0$, such that the vector bundle with connection
$(V,\, {\mathcal D})$ on $X$ is isomorphic to the base change, to $\mathbb C$, of the vector
bundle with connection $(E_0,\, D_0)$.

Similarly, the following converse of Proposition \ref{prop1}(3) will be proved.
Let $V$ be a vector
bundle on $X$ equipped with an Higgs field $\theta$. Consider the parabolic vector
bundle $f_* V$ on $\mathbb{P}^1_{\mathbb C}$ equipped with the logarithmic Higgs field
$f_* \theta$ (see \eqref{a6b}). Suppose that there is a parabolic vector bundle $W$ on
$\mathbb{P}^1_{\overline{\mathbb Q}}$, and a logarithmic Higgs field $\theta_W$ on $W$, such that
the parabolic vector bundle with logarithmic Higgs field $(f_* V,\, f_* \theta)$ is isomorphic
to the base change, to $\mathbb C$, of the parabolic vector bundle with
logarithmic Higgs field $(W,\, \theta_W)$. Then there is a vector bundle $E_0$ on $X_0$, and a
Higgs field $\theta_0$ on $E_0$, such that the Higgs bundle
$(V,\, \theta)$ on $X$ is isomorphic to the base change, to $\mathbb C$, of the Higgs
bundle $(E_0,\, \theta_0)$.

These will be done in the next section.

\section{The criterion}

\subsection{Algebraic connections}

Let
\begin{equation}\label{ee}
\mathbb{E}
\end{equation}
be a vector bundle on the curve $X_0$ in \eqref{e1}. As mentioned before, there is a
parabolic structure on the direct image $(f_0)_*\mathbb{E}\, \longrightarrow\, \mathbb{P}^1_{\overline{\mathbb Q}}$,
where $f_0$ is the map in \eqref{e1}. Consider the vector bundle $E\, \longrightarrow\, X$ in \eqref{e3}. The
parabolic vector bundle $f_*E$ on $\mathbb{P}^1_{\mathbb C}$ coincides with the base change, to $\mathbb C$,
of the above parabolic vector bundle $(f_0)_*\mathbb{E}$. Let
\begin{equation}\label{e4b}
D\ :\ E\ \longrightarrow\ E\otimes K_X
\end{equation}
be a connection on $E$. We have the logarithmic connection
\begin{equation}\label{e4}
f_*D \ :\ f_*E\ \longrightarrow\ (f_*E)\otimes K_{\mathbb{P}^1_{\mathbb C}}\otimes
{\mathcal O}_{\mathbb{P}^1_{\mathbb C}}(S_1)
\end{equation}
on $f_*E$ (see \eqref{a6}); as before, $S_1$ is the divisor $0+1+\infty$ on $\mathbb{P}^1_{\mathbb C}$.

\begin{proposition}\label{prop2}
Assume that there is a logarithmic connection
$$
D_0\ :\ (f_0)_*\mathbb{E} \ \longrightarrow\ (f_0)_*\mathbb{E}\otimes K_{\mathbb{P}^1_{\overline{\mathbb Q}}}
\otimes {\mathcal O}_{\mathbb{P}^1_{\overline{\mathbb Q}}}(S_0)
$$
(see \eqref{ee} and \eqref{e1}) whose base change, to $\mathbb C$, is the connection $f_*D$ in \eqref{e4};
as before, $S_0$ is the divisor $0+1+\infty$ on $\mathbb{P}^1_{\overline{\mathbb Q}}$.
Then there is a connection on $\mathbb{E}$ whose base change, to $\mathbb C$, coincides with the
connection $D$ (in \eqref{e4b}) on $E\,=\, \mathbb{E}\otimes_{\overline{\mathbb Q}} \mathbb{C}$ (see \eqref{e3}).
\end{proposition}

\begin{proof}
In \cite{BG} it was described how to recover $E$ (respectively, $\mathbb{E}$) from $f_*E$
(respectively, $(f_0)_*\mathbb{E}$). The proposition will be proved by examining this recovery process.

Let
\begin{equation}\label{e5}
\varphi_0\ :\ Y_0\ \longrightarrow\ \mathbb{P}^1_{\overline{\mathbb Q}}
\end{equation}
be the Galois closure of the map $f_0$ in \eqref{e1}. Let
\begin{equation}\label{e6}
\Gamma\,\, :=\,\, \text{Gal}(\varphi_0)\, \,=\,\, \text{Aut}(Y_0/\mathbb{P}^1_{\overline{\mathbb Q}})\,\,
\subset\,\, \text{Aut}(Y_0)
\end{equation}
be the Galois group for the map $\varphi_0$ in \eqref{e5}. So $\Gamma$ acts faithfully on $Y_0$ and
$Y_0/\Gamma\,=\, \mathbb{P}^1_{\overline{\mathbb Q}}$. We have a natural projection
$$
\gamma_0\ :\ Y_0\ \longrightarrow\ X_0
$$
such that $\varphi_0\,=\, f_0\circ\gamma_0$. Let
\begin{equation}\label{e7}
G \ :=\ \text{Gal}(\gamma_0)\ =\ \text{Aut}(Y_0/X_0)\ \subset\
\text{Aut}(Y_0/\mathbb{P}^1_{\overline{\mathbb Q}})
\end{equation}
be the Galois group for the above map $\gamma_0$. So $G$ is a subgroup of $\Gamma$, and $Y_0/G\,=\, X_0$.

The pullback of a parabolic vector bundle has a natural parabolic structure (see \cite[Section~3]{AB}). Consider the
vector bundle $\gamma^*_0\mathbb{E}$ on $Y_0$, where $\gamma_0$ is the map in \eqref{e7}; the direct image
$(\varphi_0)_*\gamma^*_0\mathbb{E}$, where $\varphi_0$ is the map in \eqref{e5}, has a natural parabolic structure.
Consider the pulled back parabolic vector bundle $\varphi^*_0 (\varphi_0)_*\gamma^*_0\mathbb{E}$
of the parabolic vector bundle $(\varphi_0)_*\gamma^*_0\mathbb{E}$. The parabolic structure of
$\varphi^*_0 (\varphi_0)_*\gamma^*_0\mathbb{E}$ is the trivial one, meaning there are no nonzero parabolic weights;
in fact,
\begin{equation}\label{e8}
\varphi^*_0 (\varphi_0)_*\gamma^*_0\mathbb{E}\ =\ \bigoplus_{g\in \Gamma} g^* \gamma^*_0\mathbb{E}
\end{equation}
equipped with the trivial parabolic structure (see \cite[p.~19566, Proposition 4.2(2)]{AB}, \cite[(3.9)]{BG}),
where $\Gamma$ is the Galois group in \eqref{e6}.

The base change, to $\mathbb C$, of $Y_0$ will be denoted by $Y$. Let
\begin{equation}\label{e9}
\varphi\ :\ Y\ \longrightarrow\ \mathbb{P}^1_{\mathbb C} \ \ \text{ and }\ \
\gamma\ :\ Y\ \longrightarrow\ X
\end{equation}
be the base changes, to $\mathbb C$, of the maps $\varphi_0$ and $\gamma_0$ respectively (see \eqref{e5}
and \eqref{e7}). So the base change, to $\mathbb C$, of $\varphi^*_0 (\varphi_0)_*\gamma^*_0\mathbb{E}\,
\longrightarrow\, Y_0$ is $\varphi^*\varphi_*\gamma^* E\, \longrightarrow\, Y$. Also, for any
$g\, \in\, \Gamma\,=\, \text{Gal}(\varphi_0)\,=\, \text{Gal}(\varphi)$, the vector
bundle $g^* \gamma^* E$ is the base change, to $\mathbb C$, of $g^* \gamma^*_0\mathbb{E}$. Note that the
Galois groups of $\varphi_0$ and $\varphi$ are canonically identified.

As in \eqref{e8}, we have
\begin{equation}\label{e10}
\varphi^* \varphi_*\gamma^* E \ =\ \bigoplus_{g\in \Gamma} g^* \gamma^* E
\end{equation}
equipped with the trivial parabolic structure (see \eqref{e9}). The decomposition in \eqref{e10} is the base
change, to $\mathbb C$, of the decomposition in \eqref{e8}. Note the vector bundle $\bigoplus_{g\in \Gamma} g^* \gamma^* E$
in \eqref{e10} is equipped with the connection
\begin{equation}\label{e11}
\bigoplus_{g\in \Gamma} g^* \gamma^* D,
\end{equation}
where $D$ is the connection in \eqref{e4b}. Using the isomorphism in \eqref{e10}, the connection
$\bigoplus_{g\in \Gamma} g^* \gamma^* D$ in \eqref{e11} produces a connection on $\varphi^* \varphi_*\gamma^* E$. Let
\begin{equation}\label{e12}
{\mathbb D}\ :\ \varphi^* \varphi_*\gamma^* E \ \longrightarrow\ (\varphi^* \varphi_*\gamma^* E)\otimes K_Y
\end{equation}
be this connection on $\varphi^* \varphi_*\gamma^* E$ given by $\bigoplus_{g\in \Gamma} g^* \gamma^* D$.

Consider the subgroup $G\, \subset\,\Gamma$ in \eqref{e7}. 
Fix a subset
\begin{equation}\label{ed}
\Delta\ \subset\ \Gamma
\end{equation}
such that
\begin{enumerate}
\item the composition of maps
$$
\Delta\ \hookrightarrow\ \Gamma \ \longrightarrow\ \Gamma/G
$$
is a bijection, and

\item $e\ \in\ \Delta$, where $e\, \subset\, \Gamma$ is the identity element.
\end{enumerate}
Using this subset $\Delta\, \subset\, \Gamma$, the vector bundle
$\varphi^*_0(f_0)_*{\mathbb E}\, \longrightarrow\, Y_0$ (see \eqref{e5} and \eqref{e1})
can be realized as a direct summand of the vector bundle $\varphi^*_0 (\varphi_0)_*\gamma^*_0\mathbb{E}$
in \eqref{e8}. In fact, we have
\begin{equation}\label{e13}
\varphi^*_0(f_0)_*{\mathbb E}\ =\ \bigoplus_{g\in \Delta} g^* \gamma^*_0\mathbb{E}\ \subset\
\bigoplus_{g\in\Gamma} g^* \gamma^*_0\mathbb{E}\ =\ \varphi^*_0 (\varphi_0)_*\gamma^*_0\mathbb{E}
\end{equation}
(see \eqref{e8} and \eqref{ed}); see \cite[(3.14)]{BG} for the
isomorphism in \eqref{e13}. Taking base change of \eqref{e13} to $\mathbb C$,
\begin{equation}\label{e14}
\varphi^* f_* E \ =\ \bigoplus_{g\in \Delta} g^* \gamma^* E\ \subset\ \bigoplus_{g\in \Gamma} g^* \gamma^* E
\ =\ \varphi^* \varphi_*\gamma^* E
\end{equation}
(see \eqref{e9}, \eqref{e2} and \eqref{e10}).

Consider the logarithmic connection
\begin{equation}\label{e15b}
\varphi^*_0 D_0 \ :\ \varphi^*_0(f_0)_*{\mathbb E}\ \longrightarrow\ (\varphi^*_0(f_0)_*{\mathbb E})\otimes K_{Y_0}
\otimes {\mathcal O}_{Y_0}(\varphi^*_0 S_0),
\end{equation}
where $D_0$ is the logarithmic connection on $(f_0)_*\mathbb{E}$ in the statement of the proposition. Using
the isomorphism in \eqref{e13} it produces a logarithmic connection
\begin{equation}\label{e15}
\widetilde{D}_0\ :\ \bigoplus_{g\in \Delta} g^* \gamma^*_0\mathbb{E}\ \longrightarrow\
\left(\bigoplus_{g\in \Delta} g^* \gamma^*_0\mathbb{E}\right)\otimes K_{Y_0}\otimes {\mathcal O}_{Y_0}(\varphi^*_0 S_0)
\end{equation}
on $\bigoplus_{g\in \Delta} g^* \gamma^*_0\mathbb{E}$.

Next consider the connection $\bigoplus_{g\in \Gamma} g^* \gamma^* D$ on $\bigoplus_{g\in \Gamma} g^* \gamma^* E$, where $D$
is the connection on $E$ in \eqref{e4b}. The direct summand
$\bigoplus_{g\in \Delta} g^* \gamma^* E\, \subset\, \bigoplus_{g\in \Gamma} g^* \gamma^* E$ in \eqref{e14} is evidently
preserved by this connection $\bigoplus_{g\in \Gamma} g^* \gamma^* D$. In fact, $\bigoplus_{g\in \Gamma} g^* \gamma^* D$
induces the connection $\bigoplus_{g\in \Delta} g^* \gamma^* D$ on
$\bigoplus_{g\in \Delta} g^* \gamma^* E$. Using the isomorphism in \eqref{e14} this connection
$\bigoplus_{g\in \Delta} g^* \gamma^* D$ on $\bigoplus_{g\in \Delta} g^* \gamma^* E$ produces a connection
\begin{equation}\label{e16}
\widetilde{D}\ :\ \varphi^* f_* E \ \longrightarrow\ (\varphi^* f_* E)\otimes K_Y
\end{equation}
on $\varphi^* f_* E$.

The vector bundle $\varphi^* f_* E\, \longrightarrow\, Y$ is the base change, to $\mathbb C$, of
$\varphi^*_0(f_0)_*{\mathbb E}\, \longrightarrow\, Y_0$. The connection $\widetilde{D}$ on $\varphi^* f_* E$
in \eqref{e16} is the base change, to $\mathbb C$, of the logarithmic connection $\varphi^*_0 D_0$ on
$\varphi^*_0(f_0)_*{\mathbb E}$ in \eqref{e15b}. Indeed, this follows from the given condition in the
proposition that the base change, to $\mathbb C$, of the logarithmic connection $D_0$ is the connection
$f_*D$ in \eqref{e4}. The connection $\widetilde{D}\big\vert_{Y\setminus \varphi^{-1}(\{0,1,\infty\})}$
on $(\varphi^* f_* E)\big\vert_{Y\setminus \varphi^{-1}(\{0,1,\infty\})}$ clearly coincides with the
base change, to $\mathbb C$, of the connection $(\varphi^*_0 D_0)\big\vert_{Y_0\setminus \varphi^{-1}(\{0,1,\infty\})}$
on $(\varphi^*_0(f_0)_*{\mathbb E})\big\vert_{Y_0\setminus \varphi^{-1}(\{0,1,\infty\})}$. Hence
$\widetilde{D}$ and the base change, to $\mathbb C$, of $\varphi^*_0 D_0$ coincide on entire $Y$.

Since $\widetilde{D}$ is the base change, to $\mathbb C$, of the logarithmic connection $\varphi^*_0 D_0$,
and $\widetilde{D}$ is a nonsingular connection, it follows that the logarithmic connection
$\varphi^*_0 D_0$ is actually nonsingular, meaning $\varphi^*_0 D_0$ is an usual connection.
The connection $\widetilde{D}$ on $\varphi^* f_* E\,=\, \bigoplus_{g\in \Delta} g^* \gamma^* E$ (see \eqref{e14})
preserves each direct summand $g^* \gamma^* E$, $g\,\in\, \Delta$, of $\bigoplus_{g\in \Delta} g^* \gamma^* E$. To
see this, recall that $\widetilde{D}$ is given by the connection $\bigoplus_{g\in \Delta} g^* \gamma^* D$ on
$\bigoplus_{g\in \Delta} g^* \gamma^* E$ using the isomorphism in \eqref{e14}; hence it follows from the
construction of $\widetilde{D}$ that $\widetilde{D}$ preserves the direct summand $g^* \gamma^* E$ for every
$g\,\in\, \Delta$. Hence the base change, to $\mathbb C$, of $\varphi^*_0 D_0$ preserves the
direct summand $g^* \gamma^* E$ for every $g\,\in\, \Delta$.

Recall that $e\, \in\, \Delta$. So,
the base change, to $\mathbb C$, of $\varphi^*_0 D_0$ preserves $\gamma^* E$. More
precisely, the connection $\varphi^*_0 D_0$ on $\varphi^*_0(f_0)_*{\mathbb E}\, =\,
\bigoplus_{g\in \Delta} g^* \gamma^*_0\mathbb{E}$ preserves $\gamma^*{\mathbb E}$. In other words,
$\varphi^*_0 D_0$ induces a connection on $\gamma^*{\mathbb E}$. Let
\begin{equation}\label{e17}
\widehat{\varphi^*_0 D_0}\ :\ \gamma^*_0{\mathbb E}\ \longrightarrow\ (\gamma^*_0{\mathbb E})\otimes K_{Y_0}
\end{equation}
be the connection of $\gamma^*_0{\mathbb E}$ given by $\varphi^*_0 D_0$. The base change, to $\mathbb C$,
of the connection $\widehat{\varphi^*_0 D_0}$ in \eqref{e17} evidently coincides with the connection
$\gamma^* D$ on $\gamma^* E$, because the connection $\bigoplus_{g\in \Delta} g^* \gamma^* D$ on
$\bigoplus_{g\in \Delta} g^* \gamma^* E$ produces the connection $\gamma^* D$ on $\gamma^* E$ (by
setting $g\,=\,e$).

Let $V_0$ be a vector bundle on $X_0$ and ${\mathcal D}$ a connection on the base change $V\, \longrightarrow
\, X$ of $V_0$ to $\mathbb C$. If there is a connection $\widetilde{\mathcal D}$ on $\gamma^*_0 V_0\,
\longrightarrow\, Y_0$ such that the base change, to $\mathbb C$, of $\widetilde{\mathcal D}$ coincides with
the connection $\gamma^* {\mathcal D}$ on $\gamma^* V\, \longrightarrow\, Y$, then there is a unique connection
$\underline{\widetilde{\mathcal D}}$ on $V_0\, \longrightarrow\, \mathbb X$ such that the following two statements hold:
\begin{itemize}
\item $\gamma^*_0 \underline{\widetilde{\mathcal D}}\ =\ \widetilde{\mathcal D}$, and

\item the base change, to $\mathbb C$, of $\underline{\widetilde{\mathcal D}}$ coincides with the connection
${\mathcal D}$ on $V$.
\end{itemize}
This actually follows from the fact that
$$
H^0(Y_0,\, \text{End}(\gamma^*_0 V_0)\otimes K_{Y_0})\cap (\gamma^* H^0(X,\, \text{End}(V)\otimes K_X))
\ =\ \gamma^* H^0(X_0,\, \text{End}(V_0)\otimes K_{X_0}).
$$
Another way to see this is as follows: Let $X'_0\, \subset\, X_0$ be the unique maximal open subset
over which $\gamma_0$ is unramified. Let
$$
\gamma'_0\ :\ \gamma^{-1}_0(X'_0) \ \longrightarrow\ X'_0
$$
be the restriction of the map $\gamma_0$. Then by the projection formula,
\begin{equation}\label{dc}
(\gamma'_0)_*\gamma^*_0 V_0\ =\
(V_0\big\vert_{X'_0})\otimes ((\gamma'_0)_* {\mathcal O}_{\gamma^{-1}_0(X'_0)})\ =\
V_0\big\vert_{X'_0} \oplus (V_0\big\vert_{X'_0})\otimes {\mathcal I},
\end{equation}
where ${\mathcal I}\, \subset\, (\gamma'_0)_* {\mathcal O}_{\gamma^{-1}_0(X'_0)}$ is the subbundle
whose sections over any open subset ${\mathbb U}\, \subset\, X'_0$ is the subspace of
$\Gamma(\gamma^{-1}_0({\mathbb U}),\, {\mathcal O}_{\gamma^{-1}_0({\mathbb U})})$ consisting of all
$\lambda$ such that $\sum_{y\in \gamma^{-1}_0(x)} \lambda (y)\,=\,0$ for all $x\, \in\, {\mathbb U}$.
The connection $(\gamma'_0)_* \widetilde{\mathcal D}$ on $(\gamma'_0)_*\gamma^*_0 V_0$ preserves
the decomposition in \eqref{dc}. Let $D''$ be the connection on the direct summand
$V_0\big\vert_{X'_0}$ in \eqref{dc} induced by
$(\gamma'_0)_* \widetilde{\mathcal D}$. Let $X'\, \subset\, X$ be the unique maximal open subset
over which $\gamma$ is unramified. So $X'$ is the base change, to $\mathbb C$, of $X'_0$. The base
change, to $\mathbb C$, of $V_0\big\vert_{X'_0}$ coincides with $V\big\vert_{X'}$, and the
base change, to $\mathbb C$, of the connection $D''$ on $V_0\big\vert_{X'_0}$ coincides with
${\mathcal D}\big\vert_{X'}$. Since $\mathcal D$ is a regular connection, this implies that the
connection $D''$ on $V_0\big\vert_{X'_0}$ extends to a regular connection on $V_0$, and the base
change, to $\mathbb C$, of this connection on $V_0$ is $\mathcal D$.

The above observation implies that the connection $\widehat{\varphi^*_0 D_0}$ on $\gamma^*_0{\mathbb E}$
in \eqref{e17} produces a connection
$\underline{\widehat{\varphi^*_0 D_0}}$ on ${\mathbb E}$ such that the base change, to $\mathbb C$, of
$\underline{\widehat{\varphi^*_0 D_0}}$ coincides with the connection $D$ on $E$.
\end{proof}

Combining Proposition \ref{prop1}(2) and Proposition \ref{prop2} we have the following:

\begin{theorem}\label{thm1}
Let $\mathbb E$ be a vector bundle on $X_0$, and let $D$ be a connection on the base change
$E\, \longrightarrow\, X$ of $\mathbb E$ to $\mathbb C$. Then there is a connection on $\mathbb{E}$ whose base
change, to $\mathbb C$, coincides with $D$ if and only if there is a logarithmic connection
$$
D_0\ :\ (f_0)_*\mathbb{E} \ \longrightarrow\ (f_0)_*\mathbb{E}\otimes K_{\mathbb{P}^1_{\overline{\mathbb Q}}}
\otimes {\mathcal O}_{\mathbb{P}^1_{\overline{\mathbb Q}}}(S_0)
$$
(see \eqref{ee} and \eqref{e1}) whose base change, to $\mathbb C$, is the connection $f_*D$.
\end{theorem}

\subsection{Higgs bundles}

\begin{theorem}\label{thm2}
Let $\mathbb E$ be a vector bundle on $X_0$, and let $\theta$ be a Higgs field on the base change
$E\, \longrightarrow\, X$ of $\mathbb E$ to $\mathbb C$. Then there is a Higgs field on $\mathbb{E}$ whose base
change, to $\mathbb C$, coincides with $\theta$ if and only if there is a logarithmic Higgs field
$$
\theta_0\ :\ (f_0)_*\mathbb{E} \ \longrightarrow\ (f_0)_*\mathbb{E}\otimes K_{\mathbb{P}^1_{\overline{\mathbb Q}}}
\otimes {\mathcal O}_{\mathbb{P}^1_{\overline{\mathbb Q}}}(S_0)
$$
whose base change, to $\mathbb C$, is the logarithmic Higgs field $f_*\theta$.
\end{theorem}

\begin{proof}
The proof follows exactly along the line of proof of Theorem \ref{thm1}. The modifications needed are very
straightforward. We omit the details.
\end{proof}

\section*{Acknowledgements}

The second author would like to thank Max Plank Institute, Bonn for a visiting position in May-June 2023, 
where he learnt about this question from Motohico Mulase. The authors thank him for posing the question. The 
second author would like to acknowledge the support of the SERB MATRICS grant MTR/2023/001233.
The first author is partially supported by a J. C. Bose Fellowship (JBR/2023/000003).



\begin{thebibliography}{ZZZZ}

\bibitem[AB]{AB} D. Alfaya and I. Biswas, Pullback and direct
image of parabolic connections and parabolic Higgs bundles, \textit{Int. Math. Res. Not.},
(2023), No. 22, 19546--19591.

\bibitem[At1]{At1} M. F. Atiyah, On the Krull-Schmidt theorem with application to sheaves,
{\it Bull. Soc. Math. Fr.} {\bf 84} (1956), 307--317.

\bibitem[At2]{At2} M. F. Atiyah, Complex analytic connections in fibre bundles, {\it
Trans. Amer. Math. Soc.} {\bf 85} (1957), 181--207.

\bibitem[Be]{Be} G. V. Belyi, Galois extensions of a maximal cyclotomic field, {\it Izv. Akad. 
Nauk SSSR Ser. Mat.} {\bf 43} (1979), 267--276.

\bibitem[BG]{BG} I. Biswas and S. Gurjar, A Belyi-type criterion for vector bundles on
curves defined over a number field, {\it Michigan Math. Jour.} (to appear), arXiv:2501.05681.

\bibitem[Gr]{Gr2} A. Grothendieck, Esquisse d'un programme, {\it Geometric Galois actions, 1}, 5--48,
London Math. Soc. Lecture Note Ser., 242, Cambridge University Press, Cambridge, 1997.

\bibitem[Hi]{Hi} N. J. Hitchin, Stable bundles and integrable systems, {\it
Duke Math. Jour.} {\bf 54} (1987), 91--114.

\bibitem[SL]{SL} L. Schneps and P. Lochak, \textit{Geometric Galois actions. 1.}
Around Grothendieck's ``Esquisse d'un programme'', London Math. Soc. Lecture Note Ser., 242
Cambridge University Press, Cambridge, 1997.

\bibitem[Si]{Si} C. T. Simpson, Higgs bundles and local systems,\, \textit{Inst.
Hautes \'Etudes Sci. Publ. Math.}\, \textbf{75} (1992), 5--95.

\bibitem[We]{We} A. Weil, The field of definition of a variety, {\it Amer. Jour. Math.}
{\bf 78} (1956), 509--524.

\end{thebibliography}
\end{document}